\documentclass[12pt]{article}
\usepackage{amsmath,latexsym,amsfonts,amssymb,amsthm,tipa}
\usepackage{geometry}
\usepackage{hyperref}
\usepackage{mathrsfs}
\usepackage{graphicx,color}

\newtheorem{expl}{Example}[section]
\newtheorem{defn}[expl]{Definition}
\newtheorem{prop}[expl]{Proposition}
\newtheorem{thm}[expl]{Theorem}

\newtheorem{lem}[expl]{Lemma}

\theoremstyle{definition}
\newtheorem{rem}[expl]{Remark}

\numberwithin{equation}{section}

\begin{document}

\title{Semi-continuity of Complex Singularity Exponents in Singular Central Fiber Cases}
\author{Yuchen Liu}
\date{\today}

\maketitle

\abstract{}

    In this paper, we prove the semi-continuity of complex singularity exponents for holomorphic families in singular central fiber cases. We also show that the integrals along fibers is stable for holomorphic families in some 2-dimensional cases. Besides, a sequence of counterexamples to the semi-continuity property occur for non-holomorphic families.\\


\section{Introduction}\label{sect 1}
 Complex singularity exponent is a quantitative measure of singularities of a holomorphic function. Arnold and Varchenko once studied some special cases of it from different perspectives. In \cite{tia89}, Tian first defined it clearly and applied it to his study of the Calabi problem. Related problems were also considered by Schwartz, H\"{o}rmander, Lojasiewicz, Gel'fand, Atiyah and Bernstein (see \cite{kol08}). We now recall the definition.\\

Let $X$ be a complex manifold and $g$ be a Hermitian metric on $X$.
Let $K$ be a compact subset of $X$ and $f$ be a holomorphic function
defined on $X$.

\begin{defn}
The complex singularity exponent of $f$ on
$K$ is defined to be the nonnegative number
\begin{displaymath}
c_K(f)=\sup\{c\geq 0 : \int_U|f|^{-2c}dV_g<+\infty\textrm{ for some
open neighborhood }U\textrm{ of }K\}.
\end{displaymath}
If $f(p)\neq 0$ for any $p\in K$, then we put $c_K(f)=+\infty$. If $f\equiv 0$ on $X$, then we put $c_K(f)=0$.
\end{defn}

If $K$ only contains a single point $p$, we will use the notation
$c_p(f)$ instead of $c_{\{p\}}(f)$. It is easy to see that $c_K(f)=\min_{p\in K} c_p(f)$.\\

  Lower semi-continuity is a fundamental property of complex singularity exponent. It was Varchenko who first proved the semi-continuity of complex singularity exponents for some particular holomorphic families in \cite{var83}. In \cite{tia89}, Tian has shown that the finiteness of $\int |f(x,y)|^{-2c}dV$ is stable under holomorphic perturbation of $f(x,y)$ with small sup norms, which implies semi-continuity for continuous families in 2 dimensions. For arbitrary dimensions, Phong and Sturm proved the holomorphic stability of $\int |f(z,t)|^{-2c}dV$ for 1-parameter deformations in \cite{ps00}, which implies semi-continuity for 1-parameter holomorphic families. In \cite{dk01}, Demailly and Koll\'{a}r proved semi-continuity of complex singularity exponents for plurisubharmonic functions, which contains holomorphic functions as special cases. Besides, their work generalized Tian's result about the stability of $\int |f|^{-2c}dV$ to arbitrary dimensions by applying Ohsawa-Takegoshi $L^2$ extension theorem.

\begin{thm}[\cite{dk01}, also \cite{tia89}, \cite{ps00}]\label{semicont smooth}
Let $X$ be a complex manifold, and $K$ is a compact subset of $X$.
The map $\mathcal{O}(X)\ni f\mapsto c_K(f)$ is lower semi-continuous
with respect to the topology of uniform convergence on compact sets
(uniform convergence on a fixed neighborhood of $K$ is of course
enough). More explicitly, for every nonzero holomorphic function
$f$, for every compact neighborhood $L$ of $K$ and every
$\epsilon>0$, there is a number
$\delta=\delta(f,\epsilon,K,L)>0$ such that

\begin{displaymath}
\sup_L |g-f|<\delta\quad\Rightarrow\quad c_K(g)\geq
c_K(f)-\epsilon.
\end{displaymath}

Moreover, if $c<c_K(f)$ and $g$ converges to $f$ in
$\mathcal{O}(X)$, then $|g|^{-2c}$ converges to $|f|^{-2c}$ in $L^1$
on some neighborhood $U$ of $K$.

\end{thm}

 The purpose of this paper is to study the semi-continuity property of complex singularity exponents in singular central fiber cases. As far as I know, there is no known result in these cases.
 To begin with, let $X$ be a complex
manifold with $\textrm{dim}_{\mathbb{C}}X\geq 2$, $\pi:
X\rightarrow \Delta$ be a surjective flat holomorphic map with connected fibers, where
$\Delta$ is the open disk centered at $0$ in $\mathbb{C}$ or
$\mathbb{C}$ itself. Define $X_t:=\pi^{-1}(t)$. We also suppose that
for any $t\in\Delta\setminus\{0\}$, $t$ is a regular value of $\pi$,
hence $X_t$ is smooth when $t\neq 0$.
Let $p_0$ be a singular point of $X_0$. Let $p_t\in X_t$ be a
sequence of points, such that $\lim_{t\rightarrow 0}
p_t=p_0$.\\

In this paper, we prove the following theorem:

\begin{thm}[Main Theorem]\label{main thm}
Suppose $F: X\rightarrow
\mathbb{C}$ is a holomorphic function, $X_0^{(i)}$ is any irreducible component of $X_0$ containing $p_0$. Denote $f_t=F|_{X_t}$, $f_0^{(i)}=F|_{X_0^{(i)}}$, then the inequality
\begin{displaymath}
c_{p_0}(f_0^{(i)})\leq \varliminf_{t\rightarrow 0} c_{p_t}(f_t)
\end{displaymath}
always holds.
\end{thm}

The proof of main theorem relies on the Hironaka resolution theorem. If $X_0$ is simple normal crossing, Theorem \ref{Hironaka} ensures us to choose a log resolution of $(X,V(F))$ to get a family of log resolutions of $(X_t, V_t)$, then $c_{p_t}(f_t)$ can be easily computed by these log resolutions and the adjunction formula. For general $X_0$, we first use the Hironaka resolution theorem to pull back the integral on the resolution of $X_0$, then we use the same idea as in the simple-normal-crossing case.\\

Moreover, we obtain a stronger form of the main theorem by applying the ACC for the log canonical threshold (cf. \cite{dflm09}, \cite{dflm11} and \cite{hmx12}):
\begin{thm}[Stronger form of the main theorem]\label{stronger main thm}
Suppose $F: X\rightarrow
\mathbb{C}$ is a holomorphic function, $X_0^{(i)}$ is any irreducible component of $X_0$ containing $p_0$. Denote $f_t=F|_{X_t}$, $f_0^{(i)}=F|_{X_0^{(i)}}$, then there exists $\delta>0$ such that
\begin{displaymath}
c_{p_0}(f_0^{(i)})\leq c_{p_t}(f_t)\textrm{ for any }|t|<\delta.
\end{displaymath}
\end{thm}

We also show the stability of integrals along fibers for some 2-dimensional cases.

\begin{thm}\label{eff irreducible}
Suppose $X=D_0(R_0,R_0)$, $\Delta=B_0(R_0^2)$ where $R_0>0$. Suppose $\pi$ is
defined by $\pi(x,y)=xy$. Suppose $F$ is irreducible in $\mathbb{C}\{x,y\}$. Then for any $0<c<c_0(f_0)$ and $f_t=F|_{X_{t}}$, there exists $R>0$ and $U=D_0(R,R)$ such that
\begin{displaymath}
\lim_{t\rightarrow 0} \int_{U_t}\frac{dV_t}{|f_t|^{2c}}=\int_{U_0}\frac{dV_0}{|f_0|^{2c}}.
\end{displaymath}
\end{thm}

The key idea of proof is first to examine the asymptotic behavior of the integral near $x$-axis and $y$-axis, then apply Theorem \ref{semicont smooth} and Lemma \ref{newton polygon}.
\\

When turning to non-holomorphic families, the semi-continuity property may fail to hold as the following example shows:

\begin{expl}[Counterexample in non-holomorphic families]\label{counterex 1}
Let $X=\mathbb{C}^2$, $\Delta=\mathbb{C}$,
$\pi(x,y)=xy$. Let $F(x,y)=x+y-2\sqrt{|xy|}$, $t\in \mathbb{R}_+$, $f_t(x,y)=F|_{X_{t}}$,
$p=(\sqrt{t},\sqrt{t})$, $p_0=(0,0)$. Then $c_{p_0}(f_0)=1$,
whereas $c_{p_t}(f_t)=\frac{1}{2}$. Therefore, in this case we have
\begin{displaymath}
1=c_{p_0}(f_0)>\varliminf_{t\rightarrow 0,t\in\mathbb{R}_+}
c_{p_t}(f_t)=\frac{1}{2}.
\end{displaymath}
which means that complex singularity exponents are NOT lower semi-continuous with respect to $t$. Here $F\in C_{loc}^{\frac{1}{2}}(\mathbb{C}^2)$.
\end{expl}

In Section \ref{sect 2}, we will introduce the concept of log canonical threshold and discuss its basic properties and equivalence with complex singularity exponents. The proofs of the main theorem and its stronger form will be given in Section \ref{sect 3}. In Section \ref{sect 4}, we will present the proof of Theorem \ref{eff irreducible}. Uniform upper bounds for integrals along fibers for any $F\in \mathbb{C}\{x,y\}$ will be given as a direct corollary of Theorem \ref{eff irreducible}. In Section \ref{sect 5}, we construct a sequence of non-holomorphic examples (see Example \ref{counterex n}) which do not satisfy the lower semi-continuity property besides Example \ref{counterex 1}. The key idea of construction is to find suitable homogeneous polynomials of $\sqrt{x}$ and $\sqrt{y}$. The families in Example \ref{counterex n} are $C_{loc}^{n+\frac{1}{2}}$ for arbitrary large $n$ but not $C^\infty$.

\paragraph{Acknowlegdements:} I would like to thank my thesis advisor Professor Gang Tian for giving me this problem and for his continuous support. I thank Professor Chenyang Xu for his useful comments and helps on the ACC for the log canonical threshold and related topics in algebraic geometry. I thank my teachers Professor Jinxing Cai and Professor Xiang Ma for their interest in this work and for their encouragements. I also thank my friends Yalong Shi, Alberto Della Vedova, Feng Wang and Heather Macbeth for many helpful discussions.

\section{Preliminaries on log canonical threshold}\label{sect 2}

In this section, we will mainly discuss the log canonical threshold --- an algebraic counterpart
 of the complex singularity exponent. The following discussion will
 provide us methods from algebraic geometry to compute complex singularity exponent.

The log canonical threshold is a fundamental invariant in birational geometry.
It was Shokurov who introduced it in the context of birational
geometry in \cite{sho92}. Next, we give the definition of log canonical threshold for Cartier divisors.

\begin{defn}[\cite{kol97} 8.1]
Let $X$ be a normal algebraic variety with at worst log canonical
singularities and let $D$ be an effective $\mathbb{Q}$-Cartier
divisor on $X$. The log canonical threshold of $(X,D)$ at point
$p\in X$ is defined by
\begin{displaymath}
lct_p(X,D)=\sup\{c|(X,cD) \textrm{ is log canonical in an open
neighborhood of }p\}.
\end{displaymath}
If $D=(f=0)$ then we always use the notation $lct_p(f)$ for $lct_p(X,D)$.
\end{defn}

For smooth varieties, we have an elegant equivalence between complex singularity exponent and log canonical threshold. (See Section 2 in \cite{kol08} for a detailed discussion.)

\begin{prop}[\cite{kol97} 8.2]\label{analytic LCT}
Let $X$ be a smooth variety over $\mathbb{C}$, $p\in X$ a point and
$f$ a nonzero regular function on $X$. Then

\begin{displaymath}
lct_p(f)=\sup\{c: |f|^{-2c}\textrm{ is locally $L^1$ near }p\}.
\end{displaymath}
which means $lct_p(f)=c_p(f)$.
\end{prop}

In order to compute log canonical threshold in general, we introduce the definition of \textit{relative canonical class} first.

\begin{defn}[\cite{dflm11} A.11]
Let $R$ be a commutative ring with identity. Suppose that $g: Y\rightarrow X$ is a proper birational morphism of schemes over $R$, with $Y$ nonsingular. If $r K_X$ is Cartier, then there is a unique $\mathbb{Q}$-divisor $K_{Y/X}$ supported on the exceptional locus of $g$ such that $r K_Y$ and $g^*(r K_X)+r K_{Y/X}$ are linearly equivalent. If $X$ is nonsingular, then $K_{Y/X}$ is effective and its support is the exceptional locus $Ex(g)$. We call $K_{Y/X}$ the \textit{relative canonical
divisor}  w.r.t. $g$.
\end{defn}

\begin{defn}
Let $X$ and $Y$ be complex manifolds, $g:Y\rightarrow X$ is a proper bimeromorphic morphism, then ${\rm Jac}\,g$ is a set of holomorphic functions defined on any sufficiently small holomorphic coordinate chart of $Y$, which differs by a multiple of a non-vanishing holomorphic function in the intersection of any two overlapping holomorphic coordinate charts. We define a divisor $K_{Y/X}$ on $Y$, satisfying for every sufficiently small holomorphic coordinate chart $U$, we have $K_{Y/X}\big|_U=({\rm Jac} \, g=0)$. Of course $K_{Y/X}$ is well-defined, we call it the \textit{relative canonical divisor} w.r.t. $g$.
\end{defn}

Next two propositions provide us to compute log canonical threshold by resolution of singularities.

\begin{prop}[\cite{kol97} 8.5] \label{compute LCT}
Let $X$ be a normal algebraic variety with at worst log canonical
singularities and let $f$ be a nonzero regular function on $X$. Let
$p\in X$ be a point. Let $g:Y\rightarrow X$ be a proper birational
morphism where $Y$ is smooth and $X$ is $\mathbb{Q}$-Gorenstein. We may write

\begin{displaymath}
K_{Y/X}=\sum_{i=1}^m k_i E_i \textrm{ and } V(f\circ g)=\sum_{i=1}^m
a_i E_i.
\end{displaymath}
where $E_i$ are different prime divisors on $Y$. Then we have
\begin{displaymath}
lct_p(f)\leq \min_{i: p\in g(E_i)} \frac{k_i+1}{a_i}.
\end{displaymath}
The equality holds if $\sum_i E_i$ is a simple-normal-crossing
divisor, i.e. $g:Y\rightarrow X$ is a log resolution of $(X,V(f))$. In
particular, $lct_p(f)\in \mathbb{Q}$.
\end{prop}

\begin{prop}[\cite{kol08} 11] \label{compute CSE}
Let $X$ be a complex manifold, $f$ is a nonzero regular function on $X$. Let
$p\in X$ be a point. Let $g:Y\rightarrow X$ be a proper bimeromorphic
morphism where $Y$ is a complex manifold. We may write

\begin{displaymath}
K_{Y/X}=\sum_{i=1}^m k_i E_i \textrm{ and } V(f\circ g)=\sum_{i=1}^m
a_i E_i.
\end{displaymath}
where $E_i$ are different prime divisors on $Y$. Then we have
\begin{displaymath}
c_p(f)\leq \min_{i: p\in g(E_i)} \frac{k_i+1}{a_i}.
\end{displaymath}
The equality holds if $\sum_i E_i$ is a simple-normal-crossing
divisor, i.e. $g:Y\rightarrow X$ is a log resolution of $(X,V(f))$. In
particular, $c_p(f)\in \mathbb{Q}$.
\end{prop}

We finish this section by giving some basic properties of log canonical threshold and complex singularity exponent.

\begin{expl}
Suppose $X$ is a Riemann surface, then by Proposition \ref{compute LCT} and \ref{compute CSE}
$c_p(f)=lct_p(f)=\frac{1}{\textrm{ord}_p(f)}$.
\end{expl}

\begin{expl}
Suppose $X$ is a smooth variety, $V(f)$ is a prime divisor without any singularity, then by
Proposition \ref{compute LCT} and \ref{compute CSE} we have $c_p(f)=lct_p(f)=1$ for any $p\in V(f)$.
\end{expl}

\begin{prop}[\cite{mus11} 1.6]
For any $p\in V(f)$, we have $0\leq c_p(f)=lct_p(f)\leq 1$.
\end{prop}

At the end of this section, we give the following definition of complex singularity exponents for any possibly singular analytic subvarieties:

\begin{defn}
Let $X$ be a complex manifold and $g$ be a Hermitian metric on $X$.
Suppose $W$ is a possibly singular analytic subvariety of $X$, $p\in W$ is a given point. Let $f$ be a holomorphic function defined on $W$.
Define $dV_W$ to be the volume form of the smooth locus of $W$ induced by the Hermitian metric $g$.

The complex singularity exponent of $f$ at
$p$ is defined to be the nonnegative number
\begin{displaymath}
c_p(f)=\sup\{c\geq 0 : \int_{U}|f|^{-2c}dV_W<+\infty\textrm{ for some
open neighborhood }U\subset W\textrm{ of }p\}.
\end{displaymath}
\end{defn}

By the definition above, if $W$ has irreducible components $W_0,\cdots, W_k$, then for any holomorphic function $f$ on $W$, for any $p\in V(f)$, we have

\begin{displaymath}
c_p(f)=\min_{i: p\in W_i} c_p(f|_{W_i}).
\end{displaymath}

\section{Proofs of the main theorem and its stronger form}\label{sect 3}

In this section, we will present the proofs of the main theorem \ref{main thm} and its stronger form \ref{stronger main thm}. We divide the proofs into three subsections. In Subsection \ref{sect 3.1}, we prove a special case of the main theorem, where $X_0$ is assumed to be simple-normal-crossing. In Subsection \ref{sect 3.2},  we give the complete proof of the main theorem. In Subsection \ref{sect 3.3}, we present the proof of the stronger form of the main theorem.

Before we begin, let us fix some notations throughout this section. Let $X$ be a complex
manifold with $\textrm{dim}_{\mathbb{C}}X\geq 2$, $\pi:
X\rightarrow \Delta$ be a surjective flat holomorphic map with connected fibers, where
$\Delta$ is the open disk centered at $0$ in $\mathbb{C}$ or
$\mathbb{C}$ itself. Define $X_t:=\pi^{-1}(t)$. Suppose $X_0^{(i)}$ are the irreducible components of $X_0$.
We also suppose that
for any $t\in\Delta\setminus\{0\}$, $t$ is a regular value of $\pi$,
hence $X_t$ is smooth when $t\neq 0$.
Let $p_0$ be a singular point of $X_0$. Let $p_t\in X_t$ be a
sequence of points, such that $\lim_{t\rightarrow 0}
p_t=p_0$.

\subsection{When $X_0$ is simple normal crossing}\label{sect 3.1}

In this subsection, we will always assume $X_0$ has only simple-normal-crossing singularities, and we will prove the main theorem in this special case. We will generalize this special case in the next subsection.

\begin{thm}[Main Theorem for simple-normal-crossing $X_0$]\label{main thm 1}
Suppose $F: X\rightarrow
\mathbb{C}$ is a holomorphic function, $X_0^{(i)}$ is any irreducible component of $X_0$ containing $p_0$. Suppose $X_0$ has only simple-normal-crossing singularities. Denote $f_t=F|_{X_t}$, $f_0^{(i)}=F|_{X_0^{(i)}}$, then the inequality
\begin{displaymath}
c_{p_0}(f_0^{(i)})\leq \varliminf_{t\rightarrow 0} c_{p_t}(f_t)
\end{displaymath}
always holds.
\end{thm}

The key to the proof of Theorem \ref{main thm 1} is the following fibration version of the Hironaka resolution theorem:

\begin{thm}[\cite{km98} 0.3] \label{Hironaka}
Let $f:X\rightarrow C$ be a flat morphism of a reduced algebraic
variety over $\mathbb{C}$ (or a suitably small neighborhood of a
compact set of a reduced analytic space) to a non-singular curve $C$ and
$B\subset X$ a divisor. Then there exists a projective birational
morphism $g:Y\rightarrow X$ from a non-singular $Y$ such that
$Ex(g)+g^*B+(f\circ g)^*(c)$ is an simple-normal-crossing divisor for all $c\in C$,
here $Ex(g)$ represents the exceptional divisor of $g$.
\end{thm}

Let $V:=V(F)$ be the principal divisor of $F$ on $X$. By Theorem \ref{Hironaka}
above, we may choose a complex manifold $Y$ and a bimeromorphic
morphism $g:Y\rightarrow X$ such that $Ex(g)+g^*V+(\pi\circ g)^*(t)$
is simple-normal-crossing for every $t\in \Delta$. For simplicity, denote
$\widetilde{\pi}=\pi\circ g$, $Y_t=\pi^{-1}(t)$,
$\widetilde{V}=g^*V$, $g_t=g|_{Y_t}: Y_t\rightarrow X_t$. Thus,
$Ex(g)+g^*V+Y_t$ is simple-normal-crossing for every $t\in \Delta$.

Since a divisor is a finite $\mathbb{Z}-$linear combination of prime divisors, we
may assume that both $Ex(g)$ and $\widetilde{V}$ do not contain any
irreducible component of $Y_t$ whenever $t\neq 0$ and
$|t|<\epsilon_0$ for some $\epsilon_0>0$. Denote
$V_t=V|_{X_t}$ and $\widetilde{V}_t=\widetilde{V}|_{Y_t}$.

The proof is based on the computation of $c_{p_t}(f_t)$ and $c_{p_0}(f_0)$.
Next, we shall give several lemmas to compute $c_{p_t}(f_t)$ and $c_{p_0}(f_0)$.

\begin{lem} \label{log resolution}
For any $|t|<\epsilon_0$ and $t\neq 0$, $g_t: Y_t \rightarrow
X_t$ gives a log resolution $(Y_t, \widetilde{V}_t)$ of $(X_t,
V_t)$.
\end{lem}

\begin{proof}

By definition of $Ex(g)$ we have $g: Y-Ex(g)\rightarrow X-g(Ex(g))$
is an isomorphism. By restriction of $g$ to each fiber, we obtain
isomorphisms $g_t: Y_t-(Ex(g)\cap Y_t)\rightarrow X_t-(g(Ex(g))\cap
X_t)$. We need to show $g_t:Y_t\rightarrow X_t$ is a bimeromorphic
morphism for any $|t|<\epsilon$ and $t\neq 0$.

Since $g:Y\rightarrow X$ is a bimeromorphic morphism, the
codimension of $g(Ex(g))$ in $X$ must greater than $1$, hence
$g(Ex(g))$ does not contain any irreducible component of $X_t$ for
any $t\in\Delta$. Therefore, $X_t-(g(Ex(g))\cap X_t)$ is a Zariski open
subset of $X_t$ for any $t\in\Delta$. On the other hand, since
$Ex(g)$ does not contain any irreducible component of $Y_t$, so $Y_t-(Ex(g)\cap
Y_t)$ is also Zariski open in $Y_t$. Thus, $g_t:Y_t\rightarrow X_t$
is a bimeromorphic morphism for any $|t|<\epsilon$ and $t\neq 0$.

Next we will prove this lemma. Assume $|t|<\epsilon$ and $t\neq
0$. Since $X_t$ is irreducible, $g_t$ is a bimeromorphic morphism,
we have $Y_t$ is also irreducible. Moreover, by Hironaka resolution
theorem, we have $Ex(g)+\widetilde{V}+Y_t$ is simple-normal-crossing, so $Y_t$ is
itself simple-normal-crossing, thus smooth. By the transversal intersection property, we have
$Ex(g)|_{Y_t}+ \widetilde{V}_t$ is simple-normal-crossing. Since $Ex(g_t)\subset
(Ex(g)\cap Y_t)$, $Ex(g_t)+\widetilde{V}_t$ is also simple-normal-crossing. Therefore,
$(Y_t,\widetilde{V}_t)$ is a log resolution of $(X_t, V_t)$, hence
we prove the lemma.
\end{proof}

According to the lemma above, we calculate $c_{p_t}(f_t)$ explicitly by the adjunction formula in the following lemma:

\begin{lem} \label{compute cpnfn}
Suppose $\displaystyle{\widetilde{V}=\sum_{i=1}^m a_iE_i}$,
$\displaystyle{K_{Y/X}=\sum_{i=1}^m k_i E_i}$, where $a_i, k_i$ are
integers, $\{E_i\}_{i=1}^m$ are different prime divisors on $Y$. If
$|t|<\epsilon_0$ and $t\neq 0$, then we have

\begin{displaymath}
c_{p_t}(f_t)=\min_{i:p_t\in g(E_i)}\frac{k_i+1}{a_i}.
\end{displaymath}

\end{lem}

\begin{proof}
Assume $|t|<\epsilon$ and $t\neq 0$, then we have
\begin{displaymath}
\widetilde{V}_t=\sum_{i=1}^m a_i (E_i)_t,
\end{displaymath}

where $(E_i)_t= E_i |_{Y_t}$. Next we shall calculate $K_{Y_t/X_t}$.

For any $q\in Y_t$, denote $p:=g(q)\in X_t$. Since $\pi$ is regular
at $p$, we may choose a local coordinate $(x_1,\cdots, x_n)$ of $X$
near $p$ such that $\pi(x_1,\cdots, x_n)=x_n$. Moreover, we have
$Y_t=g^*X_t$ and $Y_t$ is smooth, so $\widetilde{\pi}$ is regular at
$q$, thus we may choose a local coordinate $(y_1,\cdots, y_n)$ of
$Y$ near $q$ such that $\widetilde{\pi}(y_1,\cdots,y_n)=y_n$. When
$x_n=t$ and $y_n=t$, $(x_1,\cdots,x_{n-1})$ and $(y_1,\cdots,y_{n-1})$
give the local coordinates of $Y_t$ and $X_t$ respectively.

Under these local coordinates, we have
$g(y_1,\cdots,y_n)=(x_1,\cdots,x_{n-1}, y_n)$ since
$\widetilde{\pi}=\pi\circ g$. Therefore, we have

\begin{displaymath}
\frac{\partial(x_1,\cdots,x_{n-1}, x_n)}{\partial(y_1,\cdots,
y_{n-1},y_n)}= \left(
\begin{matrix}
\frac{\partial(x_1,\cdots, x_{n-1})}{\partial(y_1,\cdots,y_{n-1})} &
\frac{\partial(x_1,\cdots, x_{n-1})}{\partial y_n}\\
&\\ 0 & 1
\end{matrix}
\right).
\end{displaymath}

Hence we have
\begin{displaymath}
\det\left(\frac{\partial(x_1,\cdots,x_{n-1},
x_n)}{\partial(y_1,\cdots,y_{n-1},y_n)}\right)
=\det\left(\frac{\partial(x_1,\cdots,
x_{n-1})}{\partial(y_1,\cdots,y_{n-1})}\right),
\end{displaymath}
which is equivalent to ${\rm Jac} \, g={\rm Jac} \,g_t$.
Therefore, we have
\begin{align*}
K_{Y_t/X_t} & =({\rm Jac} \,g_t=0)\\
 & =({\rm Jac} \,g=0)\big|_{Y_t}\\
 & =K_{Y/X}\big|_{Y_t}.
\end{align*}

Thus we have

\begin{displaymath}
K_{Y_t/X_t}=\sum_{i=1}^m k_i (E_i)_t.
\end{displaymath}

It might happen that $(E_i)_t$ is reducible, and has several
components. However, for any $i\neq j$ since $E_i+E_j+Y_t$ is simple-normal-crossing,
so $\textrm{codim}_{Y_t}(E_i\cap E_j\cap Y_t ) \geq 2$, hence
$(E_i)_t$ and $(E_j)_t$ have no component in common, for otherwise
$\textrm{codim}_{Y_t}(E_i\cap E_j\cap Y_t)=1$.

Therefore, by Proposition \ref{compute CSE}, for any $p\in X_t$ we have
\begin{displaymath}
c_{p}(F|_{X_t})=\min_{i:p\in g(E_i\cap Y_t)}\frac{k_i+1}{a_i}.
\end{displaymath}

Hence
\begin{displaymath}
c_{p_t}(f_t)=\min_{i:p_t\in g(E_i)}\frac{k_i+1}{a_i}.
\end{displaymath}
\end{proof}

Next, let us compute $c_{p_0}(f_0)$.
Suppose $X_0=X_0^{(1)}\cup \cdots \cup X_0^{(k)}$,
$Y_0=Y_0^{(1)}\cup \cdots  \cup Y_0^{(l)}$, where $X_0^{(i_1)}$ and
$Y_0^{(j_1)}$ are different components of $X_0$ and $Y_0$.
Similarly, we have $g_0: Y_0-(Ex(g)\cap Y_0)\rightarrow
X_0-(g(Ex(g))\cap X_0)$ is isomorphism, and
$\textrm{codim}_{X_0}(g(Ex(g))\cap X_0)\geq 1$. Hence for any $1\leq
i_1\leq k$, there exists a unique $1\leq j_1\leq l$, such that
$g_0^{(j_1)}:Y_0^{(j_1)}\rightarrow X_0^{(i_1)}$ is a bimeromorphic
morphism, where $X_0^{(i_1)}$ and $Y_0^{(j_1)}$ are smooth.

Without lose of generality, we may assume $j_1=i_1$, that is, for
any $1\leq j \leq k$, $g_0^{(j)}:Y_0^{(j)}\rightarrow X_0^{(j)}$ is
a bimeromorphic morphism.

Suppose $V$ and $X_0$ have no components in common, otherwise
$c_{p_0}(f_0)=0$ and the inequality automatically holds. As a
result, $\widetilde{V}$ does not contain $Y_0^{(j)}$. Therefore, we
have

\begin{displaymath}
\widetilde{V}_0^{(j)}:=\widetilde{V}|_{Y_0^{(j)}}=\sum_{i=1}^m a_i
(E_i)_0^{(j)},
\end{displaymath}
where $(E_i)_0^{(j)}:=E_i|_{Y_0^{(j)}}$.

Using the same argument in the proof of Lemma \ref{compute cpnfn}, we
also have $(Y_0^{(j)}, \widetilde{V}_0^{(j)})$ a log resolution of
$(X_0^{(j)}, V_0^{(j)})$. Next lemma compute the relative canonical class of $(Y_0^{(j)},X_0^{(j)})$.
\\

\begin{lem} \label{compute ky0x0}
For any $1\leq j\leq k$, we have
\begin{displaymath}
K_{Y_0^{(j)}/X_0^{(j)}}=K_{Y/X}\big|_{Y_0^{(j)}}+\left(Y_0^{(j)}-g^*X_0^{(j)}\right)\big|_{Y_0^{(j)}}.
\end{displaymath}
\end{lem}

\begin{proof} With no lose of generality, we only
need to prove the lemma for $j=1$, i.e.
\begin{displaymath}
K_{Y_0^{(1)}/X_0^{(1)}}=K_{Y/X}\big|_{Y_0^{(1)}}+\left(Y_0^{(1)}-g^*X_0^{(1)}\right)\big|_{Y_0^{(1)}}.
\end{displaymath}

Assume
\begin{displaymath}
g^*(X_0^{(1)})= Y_0^{(1)}+\sum_{i=k+1}^l h_i Y_0^{(i)},
\end{displaymath}
here $h_i$ be a nonnegative integer. For any fixed $q\in Y_0^{(1)}$, assume without lose of generality that $q\notin Y_0^{(k+1)},\cdots, Y_0^{(l-s)}$, and $q\in Y_0^{(l-s+1)},\cdots, Y_0^{(l)}$ where $s\leq l-k$ is a nonnegative integer.

Since $Y_0$ is simple-normal-crossing, we may choose local coordinates $(y_1,\cdots,y_n)$ of $Y$ near $q$ such that $Y_0^{(l-i+1)}=(y_{n-i}=0)$ for any $1\leq i\leq s$ and $Y_0^{(1)}= (y_n=0)$. We may also choose local coordinates $(x_1,\cdots,x_n)$ of $X$ near $p$, and a holomorphic function $\pi_1$ near $p$ such that $X_0^{(1)}=(\pi_1=0)$. So $(y_1,\cdots, y_{n-1})$ gives a local coordinate of $Y_0^{(1)}$.

For simplicity, Denote $\widetilde{\pi}_1:=\pi_1\circ g$. Let $A$ be the matrix $\left(\frac{\partial(x_1,\cdots,x_n)}{\partial(y_1,\cdots,y_n)}\right)$, and let $B_i$ be the matrix $\left(\frac{\partial(x_1,\cdots,\widehat{x_i},\cdots,x_n)}{\partial(y_1,\cdots,y_{n-1})}\right)$.

By the formula above, in a sufficiently small neighborhood $U_q$ of $q$ we have
\begin{displaymath}
g^*(X_0^{(1)})\big|_{U_q}= Y_0^{(1)}\big|_{U_q}+\sum_{i=1}^s h_{l-i+1} Y_0^{(l-i+1)}\big|_{U_q},
\end{displaymath}
so we may choose a proper $\pi_1$ such that
\begin{displaymath}
\widetilde{\pi}_1(y_1,\cdots, y_n)=y_n\cdot\prod_{i=1}^s y_{n-i}^{h_{l-i+1}}.
\end{displaymath}

For any $1\leq i\leq n$, define $\omega$ to be a $(n-1)$-form on $(X_0^{(1)})_{\textrm{reg}}$, the regular set of $X_0^{(1)}$, as
\begin{displaymath}
\omega=(-1)^i\left(\frac{\partial \pi_1}{\partial x_i}\right)^{-1}dx_1\wedge\cdots\wedge\widehat{dx_i}\wedge\cdots \wedge dx_n.
\end{displaymath}

It is easy to check $\omega$ is well-defined, and it has no zeroes or poles near $p$ in $X_0^{(1)}$.

Hence we have
\begin{displaymath}
\left(g_0^{(1)}\right)^*\omega={\rm Jac}\, g_0^{(1)}dy_1\wedge\cdots \wedge dy_{n-1}.
\end{displaymath}

As a result, for any $1\leq i\leq n$, we have
\begin{displaymath}
{\rm Jac}\, g_0^{(1)}=(-1)^i \left(\frac{\partial\pi_1}{\partial x_i}\right)^{-1}\cdot\det B_i,
\end{displaymath}
which is equivalent to
\begin{displaymath}
\det B_i=(-1)^i\frac{\partial \pi_1}{\partial x_i}\cdot {\rm Jac}\, g_0^{(1)}.
\end{displaymath}

The Laplace expansion along the $n$-th column of $A$ yields:
\begin{align*}
{\rm Jac}\,g & =\det A\\
 & =\sum_{i=1}^n (-1)^{n-i}\frac{\partial x_i}{\partial y_n}\cdot \det B_i\\
 & =\sum_{i=1}^n (-1)^{n-i}\frac{\partial x_i}{\partial y_n}\cdot (-1)^i\frac{\partial \pi_1}{\partial x_i}\cdot {\rm Jac}\, g_0^{(1)}\\
 & =(-1)^n {\rm Jac}\, g_0^{(1)}\sum_{i=1}^n \frac{\partial \pi_1}{\partial x_i}\cdot\frac{\partial x_i}{\partial y_n}\\
 & =(-1)^n {\rm Jac}\, g_0^{(1)}\cdot\frac{\partial \widetilde{\pi}_1}{\partial y_n}\\
 & =(-1)^n {\rm Jac}\, g_0^{(1)}\cdot\frac{\partial}{\partial y_n}\left(y_n\cdot\prod_{i=1}^s y_{n-i}^{h_{l-i+1}}\right)\\
 & =(-1)^n {\rm Jac}\, g_0^{(1)}\cdot\prod_{i=1}^s y_{n-i}^{h_{l-i+1}}.
\end{align*}

Hence we have
\begin{align*}
K_{Y/X}\big|_{Y_0^{(1)}} & = ({\rm Jac}\,g=0)\big|_{Y_0^{(1)}}\\
 & = ({\rm Jac}\, g_0^{(1)}=0) + \sum_{i=1}^s h_{l-i+1}\cdot(y_{n-i}=0)\big|_{Y_0^{(1)}}\\
 & = K_{Y_0^{(1)}/X_0^{(1)}} + \sum_{i=l-s+1}^l h_i Y_0^{(i)}\big|_{Y_0^{(1)}}.
\end{align*}

Since the equation above is near $q$, we have in general
\begin{align*}
K_{Y/X}\big|_{Y_0^{(1)}} & = K_{Y_0^{(1)}/X_0^{(1)}} + \sum_{i=k+1}^l h_i Y_0^{(i)}\\
& = K_{Y_0^{(1)}/X_0^{(1)}} + \left(g^*(X_0^{(1)}) - Y_0^{(1)}\right)\big|_{Y_0^{(1)}}.
\end{align*}
which is equivalent to
\begin{displaymath}
K_{Y_0^{(1)}/X_0^{(1)}}=K_{Y/X}\big|_{Y_0^{(1)}}+\left(Y_0^{(1)}-g^*X_0^{(1)}\right)\big|_{Y_0^{(1)}}.
\end{displaymath}

Hence we prove the lemma.
\end{proof}

Based on these lemmas, to finish the proof of theorem \ref{main thm 1} we need only to compute $c_{p_0}(f_0^{(j)})$. This will be done in the following proof:

\begin{proof}[Proof of Theorem \ref{main thm 1}]
For every $1\leq j\leq k$, since $Y_0^{(j)}$ is a component of effective divisor
$g^*X_0^{(j)}$, so we have $Y_0^{(j)}- g^*X_0^{(j)}\leq 0$, hence by Lemma \ref{compute ky0x0} we have
\begin{displaymath}
K_{Y_0^{(j)}/X_0^{(j)}}\leq K_{Y/X}\big|_{Y_0^{(j)}}= \sum_{i=1}^m
k_i (E_i)_0^{(j)}.
\end{displaymath}

Suppose $p_0\in X_0^{(j)}$, then Proposition \ref{compute CSE} implies
\begin{align*}
c_{p_0}(f_0^{(j)}) & = c_{p_0}(F|_{Y_0^{(j)}})\\
 & \leq \min_{i: p_0\in g(E_i)}\frac{k_i+1}{a_i}.
\end{align*}

Thus, to prove Theorem \ref{main thm 1}, we need only to show

\begin{displaymath}
\min_{p_0\in g(E_i)}\frac{k_i+1}{a_i}\leq
\varliminf_{t\rightarrow 0} c_{p_t}(f_t).
\end{displaymath}

Since $c_{p_t}(f_t)$ belongs to the finite set $\{\frac{k_i+1}{a_i}|
1\leq i \leq m\}$, there exists $1\leq m_0\leq m$ and $t_n\rightarrow 0$ such that

\begin{displaymath}
\varliminf_{t\rightarrow 0}
c_{p_t}(f_t)=c_{p_{t_n}}(f_{t_n})=\frac{k_{m_0}+1}{a_{m_0}},
\end{displaymath}
and for any $n$ we have $p_{t_n}$ belongs to $g(E_{m_0})$.

Since $\lim_{n\rightarrow\infty}p_{t_n}=p_0$ and $g(E_{m_0})$ is closed, $p_0$ belongs to $g(E_{m_0})$. Hence

\begin{displaymath}
\min_{p_0\in g(E_i)}\frac{k_i+1}{a_i}\leq
\frac{k_{m_0}+1}{a_{m_0}}= \varliminf_{t\rightarrow 0}
c_{p_t}(f_t).
\end{displaymath}

Hence we prove Theorem \ref{main thm 1}.
\end{proof}

\subsection{Complete proof of the main theorem}\label{sect 3.2}

In the last subsection we proved the semi-continuity for complex singularity exponent when only simple-normal-crossing singularities occur in the central fiber. In this section, we will generalize this result to include all type of singularities occur in $X_0$.

In this subsection, for two continuous functions $f,g$ with non-negative values, the notation $f\succeq g$ represents that for two  if there exists $M>0$ such that $M\cdot f\geq g$; the notation $f\approx g$ represents that both $f\succeq g$ and $g\succeq f$.

We now restate the main theorem:

\begin{thm}
Suppose $F: X\rightarrow
\mathbb{C}$ is a holomorphic function, $X_0^{(i)}$ is any irreducible component of $X_0$ containing $p_0$. Denote $f_t=F|_{X_t}$, $f_0^{(i)}=F|_{X_0^{(i)}}$, then the inequality
\begin{displaymath}
c_{p_0}(f_0^{(i)})\leq \varliminf_{t\rightarrow 0} c_{p_t}(f_t)
\end{displaymath}
always holds.
\end{thm}

\begin{proof}
By Hironaka resolution theorem (\cite{hir64}), there exists a log resolution $h:(Z,Z_0)\rightarrow (X,X_0)$ such that $Z$ is smooth, $Ex(h)\subset Z_0$ and $Z_0$ is simple normal crossing. For simplicity, denote $\overline{F}:=F\circ h$, $\overline{\pi}:=\pi\circ h$, $V:=V(F)$ and $\overline{V}:=h^{*}V$.
Without loss of generality, suppose $i=1$ and $h_0^{(1)}:Z_0^{(1)}\rightarrow X_0^{(1)}$ is a bimeromorphic morphism. Since $Ex(h)\subset Z_0$, $h:Z\setminus Z_0\rightarrow X\setminus X_0$ is an isomorphism. For any $t\in \Delta-\{0\}$, denote $\overline{p}_t:=h^{-1}(p_t)$.

For any $q\in Z_0^{(1)}$, choose local coordinates of $q$ and $h(q)$, such that $h(z_1,\cdots,z_n)=(h_1,\cdots, h_n)$ and $Z_0^{(1)}=(z_1=0)$. When $z_1=0$, simple calculations yield
\begin{displaymath}
\frac{h^*dV_{X_0^{(1)}}}{dV_{Z_0^{(1)}}}\approx \sum_{i=2}^n\left|\det\left(\frac{\partial(h_1,\cdots,\hat{h_i},\cdots, h_n)}{\partial(z_2,\cdots,z_n)}\right)\right|^2=\sum_{i=1}^n |H_i|^2,
\end{displaymath}
where $H_i:=\det\left(\frac{\partial(h_1,\cdots,\hat{h_i},\cdots, h_n)}{\partial(z_2,\cdots,z_n)}\right)$.

By Laplace expansion, ${\rm Jac}\,h=\sum_{i=1}^m (-1)^{i+1}\frac{\partial h_i}{\partial z_1}\cdot H_i$, thus $V(H_1,\cdots,H_n)\subset V({\rm Jac}\, h)$. Hilbert's Nullstellensatz yields ${\rm Jac}\,h\in \sqrt{(H_1,\cdots,H_n)}$, so there exists $N\in\mathbb{N}$ and holomorphic functions $\alpha_1,\cdots,\alpha_n$ such that
\begin{displaymath}
({\rm Jac}\,h)^N=\alpha_1 H_1+\cdots+\alpha_n H_n.
\end{displaymath}

Pick $M\geq \sup(\sum_{i=1}^n |\alpha_i|^2)$, then we have
\begin{align*}
\sum_{i=1}^n |H_i|^2 & \geq \frac{1}{M}\left(\sum_{i=1}^n |H_i|^2\right)\left(\sum_{i=1}^n |\alpha_i|^2\right)\\
& \geq \frac{1}{M}\left|\sum_{i=1}^n \alpha_i H_i\right|^2 = \frac{1}{M}|{\rm Jac}\,h|^{2N}.
\end{align*}

As a result, $h^*dV_{X_0^{(1)}}\succeq |{\rm Jac}\,h|^{2N}dV_{Z_0^{(1)}}$, so for any small open neighborhood $U_0^{(1)}$ of $p_0\in X_0^{(1)}$,
\begin{equation}\label{int ineq 1}
\int_{U_0^{(1)}}\frac{dV_{X_0^{(1)}}}{|F|^{2c}}=\int_{\overline{U}_0^{(1)}}\frac{h^*dV_{X_0^{(1)}}}{|\overline{F}|^{2c}}\succeq \int_{\overline{U}_0^{(1)}}\frac{|{\rm Jac}\,h|^{2N}dV_{Z_0^{(1)}}}{|\overline{F}|^{2c}},
\end{equation}
where $\overline{U}_0^{(1)}:=h^{-1}(U_0^{(1)})$ is a small open neighborhood of $h^{-1}(p_0)$.

By Theorem \ref{Hironaka}, we may choose a complex manifold $Y$ and a bimeromorphic
morphism $g:Y\rightarrow Z$ such that $Ex(g)+g^*\overline{V}+(\overline{\pi}\circ g)^*(t)$
is simple-normal-crossing for every $t\in \Delta$. For simplicity, denote
$\widetilde{\pi}=\overline{\pi}\circ g$, $Y_t=\widetilde{\pi}^{-1}(t)$,
$\widetilde{V}=g^*\overline{V}$, $g_t=g|_{Y_t}: Y_t\rightarrow Z_t$. Without loss of generality,
assume $g_0^{(1)}:Y_0^{(1)}\rightarrow Z_0^{(1)}$ is a bimeromorphic
morphism.

Pulling back (\ref{int ineq 1}) on $Y_0^{(1)}$ yields:
\begin{equation}\label{int ineq 2}
\int_{U_0^{(1)}}\frac{dV_{X_0^{(1)}}}{|F|^{2c}}\succeq
\int_{\widetilde{U}_0^{(1)}}\frac{|{\rm Jac}\, g_0^{(1)}\cdot({\rm Jac}\,h)^N\circ g|^2dV_{Y_0^{(1)}}}{|\widetilde{F}|^{2c}},
\end{equation}
where $\widetilde{U}_0^{(1)}:=(g\circ h)^{-1}(U_0^{(1)})$.

Then
\begin{align*}
K_{Y_0^{(1)}/Z_0^{(1)}} & =({\rm Jac}\,g_0^{(1)}=0)\\
& =K_{Y/Z}|_{Y_0^{(1)}}+\left(Y_0^{(1)}-g^*Z_0^{(1)}\right)\big|_{Y_0^{(1)}}.\\
N\cdot g^*K_{Z/X}& =\left(({\rm Jac}\,h)^N\circ g=0\right),
\end{align*}

By assumption, $V({\rm Jac}\, h)=K_{Z/X}$ satisfies ${\rm supp}(K_{Z/X})\subset Ex(h)\subset Z_0$, so ${\rm supp}(V({\rm Jac}\, h))\subset Z_0$, which means ${\rm supp}(V(({\rm Jac}\, h)^N\circ g))\subset Y_0$.

Suppose $\widetilde{V}=\sum_{i=1}^m a_i E_i$, $K_{Y/Z}=\sum_{i=1}^m k_i E_i$, $Y_0^{(1)}-g^*Z_0^{(1)}+N\cdot g^*K_{Z/X}=\sum_{i=1}^m l_i E_i$, where $E_i$ are different prime divisors on $Y$. Without loss of generality, suppose $Y_0$ contains $E_1,\cdots, E_s$ but not contains $E_{s+1},\cdots, E_m$. Thus $l_i=0$ for any $i\geq s+1$, since ${\rm supp}(Y_0^{(1)}-g^*Z_0^{(1)}+N\cdot g^*K_{Z/X})\subset Y_0$.

Therefore,
\begin{eqnarray}\label{VJac}
V\left({\rm Jac}\, g_0^{(1)}\cdot({\rm Jac}\,h)^N\circ g\right) & = & V\left({\rm Jac}\, g_0^{(1)}\right)+V\left(({\rm Jac}\,h)^N\circ g\right)\nonumber\\
& = &\sum_{i=1}^m (k_i+l_i) E_i\big|_{Y_0^{(1)}}.
\end{eqnarray}

While
\begin{equation}\label{VF}
\widetilde{V}\big|_{Y_0^{(1)}}=V(\widetilde{F}\big|_{Y_0^{(1)}})=\sum_{i=1}^m a_i E_i\big|_{Y_0^{(1)}}.
\end{equation}

Combining (\ref{int ineq 2}), (\ref{VJac}) and (\ref{VF}) together, we have
\begin{align*}
c_{p_0}(f_0^{(1)})&\leq \sup\left\{c: \exists U_0^{(1)}\textrm{ s.t. }\int_{\widetilde{U}_0^{(1)}}\frac{|{\rm Jac}\, g_0^{(1)}\cdot({\rm Jac}\,h)^N\circ g|^2dV_{Y_0^{(1)}}}{|\widetilde{F}|^{2c}}<+\infty\right\}\\
& = \min_{i: h^{-1}(p_0)\cap g(E_i)\neq \emptyset}\frac{k_i+l_i+1}{a_i}.
\end{align*}

On the other hand, Lemma \ref{compute cpnfn} implies $c_{p_t}(f_t)=c_{\overline{p}_t}(\overline{F}|_{Z_t})=\min_{i: \overline{p}_t\in g(E_i)}\frac{k_i+1}{a_i}$.

Since $c_{p_t}(f_t)$ belongs to the finite set $\{\frac{k_i+1}{a_i}|
1\leq i \leq m\}$, there exists $1\leq m_0\leq m$ and $t_n\rightarrow 0$ with $t_n\neq 0$ such that

\begin{displaymath}
\varliminf_{t\rightarrow 0}
c_{p_t}(f_t)=c_{p_{t_n}}(f_{t_n})=\frac{k_{m_0}+1}{a_{m_0}},
\end{displaymath}
and for any $n$ we have $\overline{p}_{t_n}$ belongs to $E_{m_0}$. Since $t_n\neq 0$, $\overline{p}_{t_n}\not\in Y_0$, so $E_{m_0}\not\subset Y_0$, i.e. $m_0\geq s+1$.

Since $h^{-1}(p_0)$ is compact, there exists $\overline{p}_0\in h^{-1}(p_0)$ such that a subsequence of $\overline{p}_{t_n}$ converges to $\overline{p}_0$. Since $g(E_{m_0})$ is closed and contains every $\overline{p}_{t_n}$, it must also contains $\overline{p}_0$, which means $h^{-1}(p_0)\cap g(E_{m_0})\neq\emptyset$.

Therefore, $c_{p_0}(f_0^{(1)})\leq \frac{k_{m_0}+l_{m_0}+1}{a_{m_0}}$. Notice that $m_0\geq s+1$, so $l_{m_0}=0$, which means
\[
c_{p_0}(f_0^{(1)})\leq \frac{k_{m_0}+1}{a_{m_0}}=\varliminf_{t\rightarrow 0}
c_{p_t}(f_t).
\]

Hence we prove the main theorem.
\end{proof}

\subsection{Proof of the stronger form of the main theorem}\label{sect 3.3}

The key to the proof of Theorem \ref{stronger main thm} is the ACC for the log canonical threshold, where ACC stands for the ascending chain condition. The ACC for the log canonical threshold was conjectured by Shokurov in \cite{sho92}. When the dimension is three, Koll\'{a}r in \cite{kol94} proves that $1$ is not an accumulation point from below and the ACC for the log canonical threshold follows from the results of Alexeev in \cite{ale94}. Recently, the ACC Conjecture was proved by T. de Fernex, L. Ein, and M. Musta\c{t}\u{a} for complete intersections in \cite{dflm09} and even when $X$ belongs to a bounded family in \cite{dflm11}. More recently, it was proved for arbitrary varieties by C. Hacon, J. McKernan and C. Xu in \cite{hmx12}.

\begin{thm}[\cite{dflm09}, also \cite{dflm11}, \cite{hmx12}]\label{acc}
For every $n$, the set $\mathcal{T}_n$ of all complex singularity exponents $c_p(f)$, where
$f$ is a holomorphic function on an $n$-dimensional complex manifold $X$ and
$p\in V(f)$, satisfies ACC.
\end{thm}

\begin{proof}[Proof of Theorem \ref{stronger main thm}]
Suppose the conclusion is false, then we can choose a sequence $t_k\in\Delta$ with $\lim_{k\rightarrow\infty} t_k=0$, such that $c_{p_{t_k}}(f_{t_k})<c_{p_0}(f_0)$. By Theorem \ref{main thm}, we have
\begin{displaymath}
\varliminf_{k\rightarrow \infty}c_{p_{t_k}}(f_{t_k})\geq c_{p_0}(f_0).
\end{displaymath}
Thus we can choose a subsequence $\{s_l\}$ of $\{t_k\}$ such that $\{c_{p_{s_l}}(f_{s_l})\}$ is a strictly increasing sequence, which contradicts Theorem \ref{acc}.
\end{proof}

\section{Stability of integrals along fibers}\label{sect 4}

Recall that Theorem \ref{semicont smooth} not only proves the lower semi-continuity property of complex singularity exponents, but also shows the stability of integrals with respect to continuous parameters.
In this section, we present the proof of Theorem \ref{eff irreducible}, which shows the stability of integrals along fibers in some 2-dimensional cases.\\

We fix some notations throughout this section. Suppose $X=D_0(R_0,R_0)$, $\Delta=B_0(R_0^2)$ where $R_0>0$. Here $D_{(p_1,p_2)}(R_1,R_2):=\{(z,w)\in\mathbb{C}^2 : |z-p_1|<R_1, |w-p_2|<R_2\}$ is an open polydisc in $\mathbb{C}^2$, $B_p(R):=\{z\in \mathbb{C}: |z-p|<R\}$ is an open disc in $\mathbb{C}$. Suppose $\pi(x,y)=xy$ , $X_t=\pi^{-1}(t)$. Let $F$ be a holomorphic function defined on $X$ such that
$F$ is not identical to zero on each irreducible components of $X_t$ for every $t\in \Delta$. Define $f_t=F|_{X_{t}}$. Suppose $U=D_0(R,R)$ with $0<R\leq R_0$, $U_t=X_t\cap U$.

Let $g$ be the restriction of the Euclidean
metric on $\mathbb{C}^2$ to $X$, which means that $g=dx\otimes d\bar{x}+dy\otimes d\bar{y}$. Define
$X_t^*$ to be the smooth locus of $X_t$, hence $X_t^*$ is a complex
manifold, and for $t\neq 0$ we have $X_t^*=X_t$. Define $dV_t$ to be
the volume form of $g|_{X_t^*}$ on $X_t^*$.\\

The following lemma provide basic calculations about $dV_t$:

\begin{lem}
Define $dV_x=\frac{\sqrt{-1}}{2}dx\wedge d\overline{x}$, similarly
for $dV_y$. Then on $X_t=\{(x,y)\in\mathbb{C}^2\mid xy=t\}$, we have
\begin{displaymath}
dV_t=dV_x+dV_y=\frac{|x|^2+|y^2|}{|x|^2}dV_x.
\end{displaymath}
Define annulus $A(a,b):=\{z\in \mathbb{C}\mid a<|z|<b\}$. Define
\begin{displaymath}
K_t(R)  =\int_{U_t}\frac{dV_t}{|f_t|^{2c}},\quad I_t(R)
=\int_{A\left(\frac{|t|}{R}, R\right)}\frac{dV_x}{|f_t|^{2c}},\quad
J_t(R)  =\int_{A\left(\frac{|t|}{R},
R\right)}\frac{dV_y}{|f_t|^{2c}}.
\end{displaymath}
Then we have
\begin{displaymath}
K_t(R)=I_t(R)+J_t(R).
\end{displaymath}
\end{lem}

\begin{proof}
The proof is quite straight forward, so we omit it.
\end{proof}

Let $\mathbb{C}\{x,y\}$ be the ring of power series of $x,y$ in complex coefficients which is convergent in some neighborhood of $(0,0)\in \mathbb{C}^2$. Weierstrass preparation theorem implies that $\mathbb{C}\{x,y\}$ is a UFD.\\

The key to the proof of Theorem \ref{eff irreducible} is the following:

\begin{lem}[\cite{c-a00} 1.8.5] \label{newton polygon} If $f\in \mathbb{C}\{x,y\}$ is irreducible, then the Newton polygon of $f$ contains only one segment.
\end{lem}

\begin{proof}[Proof of Theorem \ref{eff irreducible}]
Since $F$ is not identically to zero on each components of $X_0$, $F$ must contain monomials of $x^k$ and $y^l$ for some $k$ and $l$.

By Weierstrass preparation theorem, there exists $h_1, h_2\in \mathbb{C}\{x,y\}$ with $h(0,0)\neq 0$ and Weierstrass polynomials $F_1(x)$ and $F_2(y)$ such that $F=h_1\cdot F_1=h_2\cdot F_2$. For sufficiently small $R$, $h_1$ and $h_2$ are both nonzero in $D_0(R,R)$, with $|h_1|$, $|h_2|$, $|h_1^{-1}|$ and $|h_2^{-1}|$ all less than $M$ in $D_0(R,R)$.

According to Lemma \ref{newton polygon}, the Newton polygon $\gamma$ of $F$ contains only one segment. Denote the two endpoints of $\gamma$ to be $(k,0)$ and $(0,l)$. It is easy to see that the Newton polygons of $F_1$ and $F_2$ are the same as $\gamma$.

Choose $s\in \mathbb{C}$ such that $t=s^{k+l}$, and $y:=t/x$, $z:=x/s^l$,
$w:=y/s^k$, so $xy=t$ and $zw=1$.

Notice that the theorem is equivalent to $\lim_{t\rightarrow 0} K_t(R)=K_0(R)$. By symmetry of $x$ and $y$, we need only to prove $\lim_{t\rightarrow 0}I_t(R)=I_0(R)$. Next, we use parameters $(z,w)$ to express $I_t(R)$.

\begin{align*}
I_t(R) & =\int_{A\left(\frac{|t|}{R},
R\right)}\frac{dV_x}{|F(x,y)|^{2c}}\\
& =\int_{A\left(\frac{|s|^k}{R}, \frac{R}{|s|^l}\right)}\frac{|s|^{2l}
dV_z}{\left|F(s^l z,s^k w)\right|^{2c}}\\
& =|s|^{2l-2klc}\int_{A\left(\frac{|s|^k}{R},
\frac{R}{|s|^l}\right)}\frac{dV_z}{\left|\frac{1}{s^{kl}}F(s^l z,s^k w)\right|^{2c}}\\
& =|s|^{2l-2klc}\int_{A\left(\frac{|s|^k}{R},
\frac{R}{|s|^l}\right)}\frac{dV_z}{\left|\frac{1}{s^{kl}}F_1(s^l z,s^k w)\right|^{2c}\cdot |h_1(s^l z, s^k w)|^{2c}}
\end{align*}

To prove the theorem, we divide the integral domain $A\left(\frac{|s|^k}{R},
\frac{R}{|s|^l}\right)$ to three separately parts $A\left(\frac{1}{R_1}, R_1\right)$, $A\left(R_1, \frac{R}{|s|^l}\right)$ and $A\left(\frac{|s|^k}{R}, \frac{1}{R_1}\right)$. Denote
\begin{align*}
I_{t,1}(R,R_1)&:=|s|^{2l-2klc}\int_{A\left(\frac{1}{R_1}, R_1\right)}\frac{dV_z}{\left|\frac{1}{s^{kl}}F(s^l z,s^k w)\right|^{2c}},\\
I_{t,2}(R,R_1)&:=|s|^{2l-2klc}\int_{A\left(R_1, \frac{R}{|s|^l}\right)}\frac{dV_z}{\left|\frac{1}{s^{kl}}F(s^l z,s^k w)\right|^{2c}},\\
I_{t,3}(R,R_1)&:=|s|^{2l-2klc}\int_{A\left(\frac{|s|^k}{R}, \frac{1}{R_1}\right)}\frac{dV_z}{\left|\frac{1}{s^{kl}}F(s^l z,s^k w)\right|^{2c}}.
\end{align*}

So $I_t(R)=I_{t,1}(R,R_1)+I_{t,2}(R,R_1)+I_{t,3}(R,R_1)$. Similarly we denote $J_{t,i}(R,R_1)$ by interchange $z$ and $w$.

We shall estimate $I_{t,1}(R,R_1)$ at first. Suppose $F_1(z,w)=\sum_{m,n=0}^{\infty} a_{m,n}z^m w^n$, let $G_1(z,w):=\sum_{lm+kn=kl} a_{m,n}z^m w^n$ be the sum of monomials of $F$ that occur in $\gamma$.

Thus,
\begin{align*}
\frac{1}{s^{kl}}F_1(s^l z,s^k w)-G_1(z,w)
& =\frac{1}{s^{kl}}F_1(s^l z,s^k w)-\frac{1}{s^{kl}}G_1(s^l z,s^k w)\\
& =\frac{1}{s^{kl}}\sum_{lm+kn\geq kl+1} s^{lm+kn} a_{m,n}z^m w^n\\
& =\frac{1}{s^{kl}}\left(s^{kl+1}\cdot H_1(s,z,w)\right)\\
& =s\cdot H_1(s,z,w).
\end{align*}

Here $H_1(s,z,w):=\sum_{lm+kn\geq kl+1} s^{lm+kn-kl-1} a_{m,n}z^m w^n$ is a holomorphic function of $s$, $z$ and $w$.

For every $R_1>1$, $H_1(s,z,w)$ is bounded for small $|s|$ and $z\in A\left(\frac{1}{R_1}, R_1\right)$. Therefore, as $s\rightarrow 0$ we have the following uniform convergence:
\begin{displaymath}
\frac{1}{s^{kl}}F_1(s^l z,s^k w)\rightrightarrows G_1(z,w)\textrm{ in }A\left(\frac{1}{R_1}, R_1\right).
\end{displaymath}

Here we treat $z$ as a single variable for $w=1/z$.

Notice that $G_1(z,w)$ is a quasi-homogeneous polynomial of $z,w$ which contains monomials $z^k$ and $w^l$, thus the order of zero point of $F_1(z,1/z)$ is at most $\max\{k,l\}$. Here we assume $c<c_0(f_0)=\frac{1}{\max\{k,l\}}$, so
\begin{displaymath}
\int_{A\left(\frac{1}{R_1}, R_1\right)}\frac{dV_z}{|G_1(z,w)|^{2c}}<+\infty.
\end{displaymath}

Therefore, we have the following equality by using Theorem \ref{semicont smooth}:
\begin{displaymath}
\lim_{s\rightarrow 0}\int_{A\left(\frac{1}{R_1},
R_1\right)}\frac{dV_z}{\left|\frac{1}{s^{kl}}F_1(s^l z,s^k w)\right|^{2c}}=
\int_{A\left(\frac{1}{R_1}, R_1\right)}\frac{dV_z}{|G_1(z,w)|^{2c}}.
\end{displaymath}

Because $2l-2klc=2l(1-kc)>0$, $\lim_{s\rightarrow 0}|s|^{2l-2klc}=0$, thus
\begin{displaymath}
\lim_{s\rightarrow 0}|s|^{2l-2klc}\int_{A\left(\frac{1}{R_1},
R_1\right)}\frac{dV_z}{\left|\frac{1}{s^{kl}}F_1(s^l z,s^k w)\right|^{2c}}= 0.
\end{displaymath}

Since $\left|F(s^l z, s^k w)\right|\leq M\cdot\left|F_1(s^l z, s^k w)\right|$, we have
\begin{align*}
\varlimsup_{s\rightarrow 0} I_{t,1}(R,R_1) &=\varlimsup_{s\rightarrow 0}|s|^{2l-2klc}\int_{A\left(\frac{1}{R_1},
R_1\right)}\frac{dV_z}{\left|\frac{1}{s^{kl}}F(s^l z,s^k w)\right|^{2c}}\\ &\leq  M^{2c}\cdot\lim_{s\rightarrow 0}|s|^{2l-2klc}\int_{A\left(\frac{1}{R_1},
R_1\right)}\frac{dV_z}{\left|\frac{1}{s^{kl}}F_1(s^l z,s^k w)\right|^{2c}}\\  &= 0.
\end{align*}

Hence we have
\begin{equation}\label{eq:it1}
\lim_{s\rightarrow 0}I_{t,1}(R,R_1)= 0.
\end{equation}

Next, we shall estimate $I_{t,2}(R,R_1)$.

Since $F_1(z,w)$ is a Weierstrass polynomial of $z$, we have $\deg_z H_1(s,z,w)\leq k-1$.

Let $H_1(s,z,w)=\sum_{i=0}^{k-1} H_{1,i}(s,w)\cdot z^i$. When $|s|<\delta_1$, $|z|\geq R_1$ for some $\delta_1>0$, $R_1>1$,
\begin{align*}
\frac{|s|\cdot|H_1(s,z,w)|}{|z|^k} & \leq |s|\sum_{i=0}^{k-1}\frac{|H_{1,i}(s,w)|}{|z|^{k-i}}\\
 &\leq \frac{\delta_1}{|z|}\sum_{i=0}^{k-1}\sup_{D_0(\delta_1,1/R_1)}|H_{1,i}(s,w)|\\
 &\leq \frac{\delta_1}{|z|}\cdot C(\delta_1, R_1).
\end{align*}

Here $C(\delta_1, R_1)$ is a constant depending on $\delta_1$ and $R_1$, and $C(\delta_1, R_1)\geq C(\delta_1', R_1')$ if $\delta_1
\geq\delta_1'$ and $R_1\leq R_1'$.

Therefore, for every $0<\epsilon_1<1$ and any fixed $\delta_1>0$, there exists $R_1$ sufficiently large such that for every $|z|\geq R_1$ and $|s|\leq \delta_1$ we have
\begin{displaymath}
\frac{|s|\cdot|H_1(s,z,w)|}{|z|^k}\leq \frac{\epsilon_1}{2}.
\end{displaymath}

On the other hand, since $G_1$ is a quasi-homogeneous polynomial of $z$ and $w$, and $G_1$ is monic of $z$, we have
\begin{displaymath}
\lim_{z\rightarrow \infty}\frac{G_1(z,w)}{z^k}=1.
\end{displaymath}

So when $R_1$ is sufficiently large, for every $|z|\geq R_1$ we have
\begin{displaymath}
1-\frac{\epsilon_1}{2}\leq \frac{|G_1(z,w)|}{|z|^k}\leq 1+\frac{\epsilon_1}{2}.
\end{displaymath}

Since $\frac{1}{s^{kl}}F_1(s^l z,s^k w)=G_1(z,w)+s\cdot H_1(s,z,w)$, we can pick $R_1$ sufficiently large and $\delta_1>0$, such that for every $|z|\geq R_1$ and $|s|\leq \delta_1$, we have
\begin{displaymath}
1-\epsilon_1\leq\frac{\left|\frac{1}{s^{kl}}F_1(s^l z,s^k w)\right|}{|z|^k}\leq 1+\epsilon_1,
\end{displaymath}
which also means
\begin{displaymath}
1-\epsilon_1\leq\frac{\left|\frac{1}{s^{kl}}F(s^l z,s^k w)\right|}{\left|h_1(s^l z, s^k w)\cdot z^k\right|}\leq 1+\epsilon_1.
\end{displaymath}

Therefore, for every $|s|\leq \delta_1$ we have
\begin{equation}\label{eq:it2a}
(1+\epsilon_1)^{-2c}\leq\frac{I_{t,2}(R,R_1)}
{|s|^{2l-2klc}\int_{A\left(R_1,
\frac{R}{|s|^l}\right)}\frac{dV_z}{\left|h_1(s^l z, s^k w)\cdot z^k\right|^{2c}}}\leq (1-\epsilon_1)^{-2c}.
\end{equation}

Notice that
\begin{displaymath}
|s|^{2l-2klc}\int_{A\left(R_1,
\frac{R}{|s|^l}\right)}\frac{dV_z}{\left|h_1(s^l z, s^k w)\cdot z^k\right|^{2c}}
=\int_{A\left(|s|^l R_1,R\right)}\frac{dV_x}{\left|h_1\left(x, \frac{t}{x}\right)\cdot G_1(x,0)\right|^{2c}}.
\end{displaymath}

If $|x|\geq |s|^l R_1$, then $|t/x|\leq |s|^k/R_1\leq |s|^k$. So there exists $\delta_2\in(0,\delta_1)$ such that for any
$ x\in A\left(|s|^l R_1, R\right)$ and $|s|\leq\delta_2$, we have
\begin{displaymath}
1-\epsilon_1\leq \frac{\left|h_1\left(x, \frac{t}{x}\right)\right|}{|h_1(x,0)|}\leq 1+\epsilon_1,
\end{displaymath}
which also means
\begin{displaymath}
1-\epsilon_1\leq \frac{\left|h_1\left(x, \frac{t}{x}\right)\cdot G_1(x,0)\right|}{|F(x,0)|}\leq 1+\epsilon_1.
\end{displaymath}

Therefore, we have
\begin{equation}\label{eq:it2b}
(1+\epsilon_1)^{-2c}\leq\frac{|s|^{2l-2klc}\int_{A\left(|s|^l R_1, R\right)}\frac{dV_z}{\left|h_1(s^l z, s^k w)\cdot z^k\right|^{2c}}}
{\int_{A\left(|s|^l R_1, R\right)}\frac{dV_x}{|F(x,0)|^{2c}}}\leq (1-\epsilon_1)^{-2c}.
\end{equation}

Combining (\ref{eq:it2a}) and (\ref{eq:it2b}) together, we have
\begin{displaymath}
(1+\epsilon_1)^{-4c}\leq\frac{I_{t,2}(R,R_1)}{\int_{A\left(|s|^l R_1, R\right)}\frac{dV_x}{|F(x,0)|^{2c}}}
\leq (1-\epsilon_1)^{-4c}.
\end{displaymath}

It is easy to see that
\begin{displaymath}
\lim_{s\rightarrow 0}\int_{A\left(|s|^l R_1, R\right)}\frac{dV_x}{|F(x,0)|^{2c}}=\int_{A(0, R)}\frac{dV_x}{|F(x,0)|^{2c}}=I_0(R).
\end{displaymath}

Therefore, we have the following two inequalities.
\begin{equation}\label{eq:it2g}
\varliminf_{s\rightarrow 0} I_{t,2}(R,R_1)\geq (1+\epsilon_1)^{-4c}I_0(R).
\end{equation}
\begin{equation}\label{eq:it2l}
\varlimsup_{s\rightarrow 0} I_{t,2}(R,R_1)\leq (1-\epsilon_1)^{-4c}I_0(R).
\end{equation}

Finally, we shall estimate $I_{t,3}(R,R_1)$. We have the following equality:
\begin{align*}
I_{t,3}(R,R_1) & = |s|^{2l-2klc}\int_{A\left(\frac{|s|^k}{R}, \frac{1}{R_1}\right)}\frac{dV_z}{\left|\frac{1}{s^{kl}}F(s^l z,s^k w)\right|^{2c}}\\
& = |s|^{2l-2klc}\int_{A\left(R_1, \frac{R}{|s|^k}\right)}\frac{dV_w}{\left|\frac{1}{s^{kl}}F(s^l z,s^k w)\right|^{2c}\cdot|w|^4}
\end{align*}

We define $G_2(z,w)$ and $H_2(s,z,w)$ similarly. Because of the same reason in the discussion about $I_{t,1}(R,R_1)$, for $\epsilon_1>0$ there exists
sufficiently large $R_1$ such that for any $w\in A\left(R_1, \frac{R}{|s|^k}\right)$, we have
\begin{displaymath}
1-\epsilon_1\leq\frac{\left|\frac{1}{s^{kl}}F(s^l z,s^k w)\right|}{\left|h_2(s^l z, s^k w)\cdot w^l\right|}\leq 1+\epsilon_1.
\end{displaymath}

So for any $w\in A\left(R_1, \frac{R}{|s|^k}\right)$, we have
\begin{align*}
\left|\frac{1}{s^{kl}}F(s^l z,s^k w)\right| & \geq (1-\epsilon_1)\left|h_2(s^l z, s^k w)\right|\cdot |w|^l\\
& \geq \frac{1-\epsilon_1}{M}\cdot |w|^l.
\end{align*}

Therefore,
\begin{align*}
I_{t,3}(R,R_1)
& = |s|^{2l-2klc}\int_{A\left(R_1, \frac{R}{|s|^k}\right)}\frac{dV_w}{\left|\frac{1}{s^{kl}}F(s^l z,s^k w)\right|^{2c}\cdot|w|^4}\\
& \leq |s|^{2l-2klc}\int_{A\left(R_1, \frac{R}{|s|^k}\right)}\frac{dV_w}{\left(\frac{1-\epsilon_1}{M}\cdot |w|^l\right)^{2c}\cdot|w|^4}\\
& =\frac{M^{2c}}{(1-\epsilon_1)^{2c}}|s|^{2l-2klc}\int_{A\left(R_1, \frac{R}{|s|^k}\right)}\frac{dV_w}{|w|^{2lc+4}}\\
& =\frac{M^{2c}}{(1-\epsilon_1)^{2c}}|s|^{2l-2klc}\int_{0}^{2\pi}d\theta\int_{R_1}^{\frac{R}{|s|^k}}r^{-2lc-4}\cdot rdr\\
& \leq\frac{2\pi M^{2c}}{(1-\epsilon_1)^{2c}}|s|^{2l-2klc}\int_{R_1}^{+\infty}r^{-2lc-3}dr\\
& =\frac{2\pi M^{2c}}{(1-\epsilon_1)^{2c}}\cdot \frac{R_1^{-2lc-2}}{2lc+2} |s|^{2l-2klc}.
\end{align*}

Notice that $2l-2klc=2kl(\frac{1}{k}-c)>0$, so
\begin{displaymath}
\varlimsup_{s\rightarrow 0}I_{t,3}(R,R_1)\leq \lim_{s\rightarrow 0}\frac{2\pi M^{2c}}{(1-\epsilon_1)^{2c}}\cdot \frac{R_1^{-2lc-2}}{2lc+2} |s|^{2l-2klc}=0,
\end{displaymath}
which means
\begin{equation}\label{eq:it3}
\lim_{s\rightarrow 0}I_{t,3}(R,R_1)=0.
\end{equation}

Combining (\ref{eq:it1}), (\ref{eq:it2g}), (\ref{eq:it2l}) and (\ref{eq:it3}) together, we have
\begin{eqnarray}\label{eq:itleps}
\varlimsup_{s\rightarrow 0}I_t(R) & = &\sum_{i=1}^3\varlimsup_{s\rightarrow 0}I_{t,i}(R,R_1)\nonumber\\
& = &\varlimsup_{s\rightarrow 0}I_{t,2}(R,R_1)\nonumber\\
& \leq & (1-\epsilon_1)^{-4c}I_0(R).
\end{eqnarray}

Similarly,
\begin{equation}\label{eq:itgeps}
\varliminf_{s\rightarrow 0}I_t(R)\geq (1+\epsilon_1)^{-4c}I_0(R).
\end{equation}

When $\epsilon_1\rightarrow 0$, the limits of (\ref{eq:itleps}) and (\ref{eq:itgeps}) are the following:
\begin{displaymath}
I_0(R)\leq \varliminf_{s\rightarrow 0}I_t(R)\leq \varlimsup_{s\rightarrow 0}I_t(R)\leq I_0(R),
\end{displaymath}
which means
\begin{displaymath}
\lim_{t\rightarrow 0} I_t(R)=\lim_{s\rightarrow 0} I_t(R)=I_0(R).
\end{displaymath}

Hence we finish the proof.
\end{proof}

At the end of this section, we may show the existence of uniform upper bounds for integrals along fibers as a direct corollary of Theorem \ref{eff irreducible}.

\begin{thm}\label{eff dim 2}
For any $F\in \mathbb{C}\{x,y\}$ and any $0<c<c_0(f_0)$, there exists $U=D_0(R,R)$, $\delta>0$
and $M=M(R,F,c,\delta)>0$, such that the inequality
\begin{displaymath}
\int_{U_t}\frac{dV_t}{|f_t|^{2c}}\leq M
\end{displaymath}
holds for all $|t|\leq \delta$.
\end{thm}

\begin{lem}\label{eff reducible}
If Theorem \ref{eff dim 2} is true for $F_1$ and $F_2$, then it is also true for $F:=F_1\cdot F_2$.
\end{lem}

\begin{proof}
Suppose $c_0(F_i)=\frac{1}{l_i}$, then the multiplicity of $F$ at $0$ when restricts to $X_0$ is at least $l:=l_1+l_2$, thus $c_0(F\cdot G)\leq \frac{1}{l}$.

For every $c<\frac{1}{l}\leq c_0(F)$, by Young's inequality we have
\begin{displaymath}
|F|^{-2c}\leq \frac{l_1}{l}|F_1|^{-2c\cdot\frac{l}{l_1}}+\frac{l_2}{l}|F_2|^{-2c\cdot\frac{l}{l_2}}.
\end{displaymath}

Here $c_i=c\cdot\frac{l}{l_i}<c_0(F_i)$.

Hence for every $|t|<\min(\delta_1,\delta_2)$,

\begin{align*}
\int_{U_t}\frac{dV_t}{|F|^{2c}} & \leq \frac{l_1}{l}\int_{U_t}\frac{dV_t}{|F_1|^{2c_1}}+\frac{l_2}{l}\int_{U_t}\frac{dV_t}{|F_2|^{2c_2}}\\
 & \leq \frac{l_1}{l}\cdot M_1+\frac{l_2}{l}\cdot M_2.
\end{align*}

Define $M:=\frac{l_1}{l}\cdot M_1+\frac{l_2}{l}\cdot M_2$, hence Theorem \ref{eff dim 2} is true for $F$.
\end{proof}

\begin{proof}[Proof of Theorem \ref{eff dim 2}]
Let $\mathcal{S}=\{F\in \mathbb{C}\{x,y\}\mid \textit{Theorem \ref{eff dim 2} is true for }F\}$.

By Theorem \ref{eff irreducible}, $\mathcal{S}$ contains all irreducible elements of $\mathbb{C}\{x,y\}$; by Lemma \ref{eff reducible}, $\mathcal{S}$ is closed under multiplication. Since $\mathbb{C}\{x,y\}$ is a UFD, we have $\mathcal{S}=\mathbb{C}\{x,y\}$, which means that Theorem \ref{eff dim 2} is true for every $F\in \mathbb{C}\{x,y\}$.
\end{proof}

\begin{rem}
The proof of Theorem \ref{eff irreducible} cannot be generalized to higher dimensions directly, since Lemma \ref{newton polygon} is usually false in higher dimensions.
\end{rem}

\section{Counterexamples in non-holomorphic families}\label{sect 5}

In Section \ref{sect 3}, we prove the semi-continuity of complex singularity exponents for holomorphic families.
However, Example \ref{counterex 1} shows that this is not true for non-holomorphic families.
Moreover, we construct a sequence of counterexamples (see Example \ref{counterex n}) to the semi-continuity property for non-holomorphic families in this section. The families in Example \ref{counterex n} are $C_{loc}^{n+\frac{1}{2}}$ for arbitrary large $n$ but not $C^\infty$.\\

For every $n\in\mathbb{N}$, define $V_n:=\{q(z^2)+c\cdot
z^{2n+1}\mid q(z)\in \mathbb{C}[z],\, \deg q \leq 2n+1,\,
c\in\mathbb{C}\}$, then $V_n$ is a complex linear subspace of
$\mathbb{C}[z]$ with $\dim V_n=2n+3$.
Define $W_n:=\{P(z)\in V_n: (z-1)^{2n+2}\mid P(z)\}$. We have
$W_n\neq 0$ because $\dim V_n=2n+3>2n+2=\deg (z-1)^{2n+2}$.

\begin{lem}\label{polynomial}
Let $\mathcal{N}=\{n\in \mathbb{N}: W_n$ contains a polynomial $
P_n(z)$ with the $z^{4n+2}$ term and constant term both nonzero.$\}$,
then $\mathcal{N}$ is an infinite set.
\end{lem}

\begin{proof}
Suppose the conclusion is not true, then there exists
$N\in\mathbb{N}$ such that for every $n\geq N$ and every nonzero
$P_n(z)\in W_n$, at least one of $z^{4n+2}$ term and constant term
of $P_n$ vanishes.

Let us fix $n>N$ sufficiently large. If the constant term of $P_n$ is nonzero and the $z^{4n+2}$ term of
$P_n$ is zero, let $\widetilde{P_n}(z):=P_n(z)+z^{4n+2}\cdot
P_n(1/z)$, then of course $(z-1)^{2n+2}\mid \widetilde{P_n}(z)$,
thus $\widetilde{P_n}(z)\in W_n$ with $z^{4n+2}$ term and constant
term both nonzero, which contradicts to our assumption. Hence the
constant term of $P_n$ must vanish, so does the $z^{4n+2}$ term.

Because the $z^{4n+2}$ term and constant term of $P_n$ both
vanishes, $P_n(z)/z^2$ must belong to $W_{n-1}$, then we may denote
$P_{n-1}(z):=P_n(z)/z^2$. After same argument applying to $P_{n-1}$,
we will get $P_{n-2}(z):=P_{n-1}(z)/z^2\in W_{n-2}$. Repeat this process to
$P_{n-k}$ for $k=n-N$, finally we have that $P_n(z)=z^{2k}\cdot
P_{n-k}(z)$ and $P_{n-k}\in W_{n-k}$.

Since $(z-1)^{2n+2}\mid \widetilde{P_n}(z)$, we have
$(z-1)^{2n+2}\mid P_{n-k}(z)$, thus $2n+2\leq\deg P_{n-k}\leq 4N+2$,
i.e. $n\leq 2N$, which contradicts to our assumption that $n$ can be
sufficiently large. Hence the lemma is proved.
\end{proof}

Next, we construct a sequence of counterexamples to the semi-continuity property for non-holomorphic families.

\begin{expl}\label{counterex n}
Suppose $n\in\mathcal{N}$. Let $P_n(z)=q_n(z^2)+c_n\cdot z^{2n+1}$. Define
$Q_n(x,y):=q_n(x/y)\cdot y^{2n+1}$, $F_n(x,y)=Q_n(x,y)+c_n\cdot
|xy|^\frac{2n+1}{2}$. Then for every $t\in \mathbb{R}_+$ and
$s=\sqrt t\in \mathbb{R}_+$, we have
\begin{displaymath}
c_{(s,s)}(F_n|_{X_t})\leq
\frac{1}{2n+2}<\frac{1}{2n+1}=c_{(0,0)}(F_n|_{X_0}).
\end{displaymath}

Therefore,
\begin{displaymath}
\frac{1}{2n+1}=c_{(0,0)}(F_n|_{X_0})>\frac{1}{2n+2}\geq\varliminf_{s\rightarrow 0}c_{(s,s)}(F_n|_{X_t}).
\end{displaymath}

Thus, the semi-continuity property of complex singularity exponents of these families fail to hold.
\end{expl}

\begin{proof}
We shall estimate $c_{(s,s)}(F_n|_{X_t})$ first.

\begin{align*}
F_n\left(x,\frac{t}{x}\right) & = Q_n\left(x,\frac{t}{x}\right)+c_n\cdot t^\frac{2n+1}{2}\\
 & =
 q_n\left(\frac{x^2}{t}\right)\cdot\left(\frac{t}{x}\right)^{2n+1}+c_n\cdot
 t^\frac{2n+1}{2}\\
 & = q_n\left(\left(\frac{x}{s}\right)^2\right)\cdot\frac{s^{4n+2}}{x^{2n+1}}+c_n\cdot
 s^{2n+1}\\
 & = \frac{s^{4n+2}}{x^{2n+1}}\cdot
 \left(q_n\left(\left(\frac{x}{s}\right)^2\right)+c_n\cdot\left(\frac{x}{s}\right)^{2n+1}\right)\\
 & = \frac{s^{4n+2}}{x^{2n+1}}\cdot p_n\left(\frac{x}{s}\right).
\end{align*}

Thus, we have
\begin{align*}
c_{(s,s)}(F_n|_{X_t}) & = \frac{1}{ord_{x=s}(F_n(x,t/x))}\\
& = \frac{1}{ord_{x=s} (p_n(x/s))}\\
& \leq \frac{1}{2n+2}.
\end{align*}

On the other hand, $F_n(x,y)=Q_n(x,y)$ is a homogenous polynomial
with degree $2n+1$ when $xy=0$. Besides, $x^{2n+1}$ term and
$y^{2n+1}$ term both occur in $Q_n(x,y)$. Consequently,
$c_{(0,0)}(F_n|_{X_0})=\frac{1}{2n+1}$, hence we finish our proof.
\end{proof}

Next proposition shows that for any $n\in\mathcal{N}$, $F_n\in C_{loc}^{n,\frac{1}{2}}(\mathbb{C}^2)$ and $F_n
\notin C^\infty(\mathbb{C}^2)$, hence $F_n$ is not holomorphic.

\begin{prop}
For every $n\in\mathcal{N}$, $F_n\in C_{loc}^{n,\frac{1}{2}}(\mathbb{C}^2)$ while $F_n\notin C^\infty(\mathbb{C}^2)$.
\end{prop}

\begin{proof}
We notice that $Q_n(x,y)$ is a homogeneous polynomial of $x,y$, so
$Q_n\in C^\infty (\mathbb{C}^2)$. Thus to prove the proposition it
is sufficient to prove that $|xy|^\frac{2n+1}{2}\in
C_{loc}^{n,\frac{1}{2}}(\mathbb{C}^2)$.

It can be easily verified that $|z|^{n+\alpha}\in
C_{loc}^{n,\alpha}(\mathbb{C})$ for every $n\in\mathbb{N}$ and
$0<\alpha<1$, so $|z|^{\frac{2n+1}{2}}\in
C_{loc}^{n,\frac{1}{2}}(\mathbb{C})$. By composing
$|z|^{\frac{2n+1}{2}}$ with the $C^\infty$ function $z=xy$, we have
$|xy|^\frac{2n+1}{2}\in C_{loc}^{n,\frac{1}{2}}(\mathbb{C}^2)$. It
can also be easily verified that $|xy|^{n+\alpha}\notin C^\infty(\mathbb{C}^2)$, thus
we prove the proposition.
\end{proof}

Since $\mathcal{N}$ is an infinite set, for any $N\in\mathbb{N}$ there exists $n\in \mathcal{N}$ and $n>N$ such that $F_n\in C_{loc}^N(\mathbb{C}^2)$. Thus the families in Example \ref{counterex n} are $C_{loc}^{n+\frac{1}{2}}$ for arbitrary large $n$ but not $C^\infty$.

School of Mathematical Sciences, Peking University,

Beijing, 100871, P. R. China.

Email: jerry.liu1022@gmail.com
\end{document}